\documentclass[12 pt]{amsart}
\usepackage{amscd,amsfonts,amssymb,amsmath}
\usepackage{hyperref}
\usepackage[cp1251]{inputenc}
\usepackage[T2A]{fontenc}

\usepackage{epsfig}
\newtheorem{theorem}{Theorem}[section]
\newtheorem{corollary}[theorem]{Corollary}
\newtheorem{lemma}[theorem]{Lemma}
\newtheorem{proposition}[theorem]{Proposition}
\theoremstyle{definition}

\newtheorem{question}[theorem]{Question}

\newtheorem{definition}[theorem]{Definition}
\newtheorem{example}[theorem]{Example}
\newtheorem{remark}[theorem]{Remark}
\numberwithin{equation}{subsection}
\newtheorem*{ack}{Acknowledgement}
\usepackage[all,cmtip]{xy}
\usepackage{xcolor}
\usepackage{xfrac}
\usepackage{pb-diagram}
\usepackage{graphicx}
\usepackage{float}

\newcommand{\VAP}{\mathrm{VAP}}
\newcommand{\AB}{\mathrm{AB}}

\newcommand{\Aut}{\operatorname{Aut}}
\newcommand{\Sym}{\operatorname{Sym}}

\newcommand{\id}{\mathrm{id}}

\def\bea{\begin{eqnarray}}
\def\eea{\end{eqnarray}}
\def\nn{\nonumber}

\begin{document}
\title{Extensions of Yang-Baxter sets}
\author{Valeriy G. Bardakov}
\author{Dmitry V. Talalaev}

\date{\today}
\address{Sobolev Institute of Mathematics, Novosibirsk 630090, Russia.}
\address{Novosibirsk State Agrarian University, Dobrolyubova street, 160, Novosibirsk, 630039, Russia.}
\address{Regional Scientific and Educational Mathematical Center of Tomsk State University,
36 Lenin Ave., 14, 634050, Tomsk, Russia.}
\email{bardakov@math.nsc.ru}

\address{Moscow State University, 
119991 Moscow RUSSIA }
\email{dtalalaev@yandex.ru}

\subjclass[2010]{Primary 17D99; Secondary 57M27, 16S34, 20N02}
\keywords{Yang--Baxter equation, set-theoretic solution, quandle, virtual pure braid group, Hopf algebra }

\begin{abstract}
The paper extends the notion of braided set and its close relative - the Yang-Baxter set – to the category of vector spaces and explore structure aspects of such a notion as morphisms and extensions. In this way we describe a family of solutions for the Yang--Baxter equation  on  $B \otimes C$ (on  $B \times C$, respectively) if given $(B, R^B)$ and $(C, R^C)$ are two linear (set-theoretic) solutions of the Yang--Baxter equation. One of the key observation is the relation of this question with the virtual pure braid group.

\end{abstract}
\maketitle
\tableofcontents

\section{Introduction}

A solution of the quantum Yang--Baxter equation (YBE) is a linear map $R : V \otimes V \to V\otimes V$ satisfying 
\begin{equation}
\label{YB}
	R_{12} R_{13} R_{23} = R_{23} R_{13} R_{12},
\end{equation}
 where  $V$ is a vector space over a field $K$ and $R_{ij} : V \otimes V \otimes V  \to  V \otimes V \otimes V$ acts as $R$ on the $(i,j)$ tensor factor and as the identity on the remaining factor. 
V.~M.~Buchstaber~\cite{Bukh} called the map $R$ which satisfies the YBE by a  {\it Yang-Baxter map}. The pair $(V, R)$ is said to be a {\it solution of the YBE} or simply a {\it solution}. 

The  Yang-Baxter equation, or the 2-simplex equation, or the triangle equation, is one of the basic  equations in mathematical physics and in  low dimension topology.
It lies in the foundation of the theory of quantum groups, solvable models of statistical mechanics,   knot theory, braid theory. 
At first, the YBE appears  in the paper of C.~N.~Yang  \cite{Yang} for studying the many-body problem. Later R.~J.~Baxter \cite{Bax} introduced  this equation for the study of solvable vertex models in statistical mechanics as the condition of commuting transfer matrices. Another derivation of the YBE follows from the factorization of the $S$-matrix in $1+1$ dimensional Quantum Field Theory (see papers of  A.~B.~Zamolodchikov \cite{Z,Z-1}). The YBE is also essential in Quantum Inverse Scattering Method for integrable systems \cite{STF, TF}.

V.~G.~Drinfeld \cite{Drinfeld} suggested focusing on a specific class of solutions: the set-theoretic, i.~e. solutions for which the vector space $V$ is  spanned by a set $X$, and $R$ is the linear operator induced by a map $R: X \times X \to  X \times X$. In this case we say that $(X,R)$ is a set-theoretic solution to the Yang--Baxter equation or simply a solution to the YBE.  We also call the pair $(X,R)$ the Yang-Baxter set. Set-theoretic solutions have connections for example with groups of I-type, Bieberbach groups, bijective 1-cocycles, Garside theory and a wide class of integrable discrete dynamical systems \cite{Veselov, Ves2, BS}.

It is easy to see that for any $X$ the map $P(x, y) = (y, x)$ gives a set-theoretic solution for the YBE. On the other side, if $R$ is a solution for the YBE, then the map $S = P R$ satisfies the braid relation
$$
(S \times \id) (\id \times S) (S \times \id) = (\id \times S) (S \times \id) (\id \times S)
$$
- the defining relation in the braid group $B_n$.  Topologically the braid relation is simply the third Reidemeister move of planar diagrams of links. 
In 1980s D.~Joyce \cite{Joyce} and S.~V.~Matveev~\cite{Mat} introduced quandles as invariants of knots and links and proved that any quandle gives an elementary set-theoretic solution to the braid equation.

It is well known (see, for example, \cite[Chapter 10, Section 6]{Kas}) that  any braided set $(X, S)$ with invertible $S: X \times X \to X \times X$ defines a representation 
$B_n \to \Sym(X^n)$, and the composition $R=PS$ gives an invertible solution for the YBE.
We show that this  solution $(X, R)$   defines a representation of the  virtual pure braid group $VP_n$, for any $n \geq 2$,  into $\Sym(X^{n})$.

In this paper, we develop such a point of view on solutions for the Yang-Baxter equation or on braidings as a representative of a PROP structure with biarity (2,2) morphisms \cite{Markl}. Let us recall that PROP is a generalization of the concept of an operad with operations of the highest valency, in particular, an operation that brings a pair of values for given two arguments. These structures are of great importance in the study of multivalued groups \cite{BR}. Algebras over PROPs with biarity (2,2) morphisms interpolate between algebras and coalgebras. Apparently, this is why the Yang-Baxter equation plays such a significant role in the theory of Lie-bialgebras and Hopf algebras. In this context the following questions are of great importance: the functors and equivalences between such PROPs, the extensions of such categories and possible classifications. It turns out that in contrast to the notion of extension which is natural in the category of groups here the notion close to the bicrossed product of groups is more general and meaningful. The main result of the paper is related to this. First of all, we investigate special subgroups of the virtual braid group $VP_n$, construct extensions of braided sets as representations of subgroups of $VP_n$. We also develop the similar formalism in the vector space category, in the case of extensions of quasi-triangular Hopf algebras.

The paper is organized as follows. In Section \ref{sec-prelim} we recall known facts on Hopf algebras, extensions of braided sets related to group structures and the virtual pure braid group. In Section \ref{sec-Br} we establish connection between solutions for the Yang-Baxter equation and representation of the virtual pure braid group and the analog for the braid equation. In Sections \ref{sec-Ext} and \ref{sec-Ext2} we elaborate the extension procedure in the Hopf algebra and the set-theoretic case respectively. The Section \ref{sec-Gr} is devoted to the description of the group controlling our extension procedure and  interpret our construction in terms of simplicial sets and simplicial groups.

\section{Preliminaries }\label{sec-prelim}

\subsection{Hopf algebra} Recall some definitions in the theory of Hopf algebras (see, for example, \cite{Kas}). Suppose that $H$ is a vector space  over a field $K$. {\it Comultiplication} on
$H$ is a linear map
$$
\Delta : H \to H \otimes_K H.
$$
For $H$ of finite dimension a multiplication on $H$ induces a comultiplication on $H^*$ and vice versa - a comultiplication on $H$ induces a multiplication on $H^*$. Hence $H^*$ becomes an algebra which is not-associative in general.

Let us denote a comultiplication  of $h \in H$ by
$$
\Delta(h) = \sum h_i^{(1)} \otimes h_i^{(2)}.
$$
Comultiplication $\Delta$ is said to be {\it cocommutative}, if
$$
\Delta(h) = \sum h_i^{(2)} \otimes h_i^{(1)}.
$$
Comultiplication $\Delta$ is said to be {\it coassociative} if the following holds in $H \otimes H \otimes H$:
$$
\sum \Delta \left( h_i^{(1)} \right) \otimes h_i^{(2)} = \sum h_i^{(1)} \otimes \Delta \left( h_i^{(2)} \right) \Leftrightarrow (\Delta \otimes id_H) \Delta = (id_H \otimes  \Delta)  \Delta.
$$

{\it Counit} is a functional $\varepsilon : H \to K$  such that
$$
\sum h_i^{(1)}  \varepsilon \left( h_i^{(2)} \right) = \sum \varepsilon \left( h_i^{(1)} \right)  h_i^{(2)},
$$
where
$$
\Delta(h) = \sum h_i^{(1)} \otimes h_i^{(2)}.
$$
This axiom can be presented in the form
$$
(\varepsilon \otimes id_H) \Delta = \id_H = (id_H \otimes  \varepsilon)  \Delta.
$$

\begin{definition}
Let $H$ be a vector space  over a field $K$. If there exist coassociative comultiplication $\Delta$ and counit $\varepsilon$ on $H$, then $(H, \Delta, \varepsilon)$ is said to be {\it coalgebra}.

{\it Bialgebra} $H$ is an associative algebra with unit on which it is defined a coassociative comultiplication
$$
\Delta : H \to H \otimes H
$$
and counit that are homomorphisms of $K$-algebras.
\end{definition}

An {\it antipode} of the bialgebra $H$ is an anti-homomorphism
$$
S : H \to H
$$
such that
$$
\sum h_i^{(1)} \cdot S \left( h_i^{(2)} \right) = \sum S \left( h_i^{(1)} \right) \cdot h_i^{(2)} = \varepsilon(h) \cdot 1,
$$
where
$$
\Delta(h) = \sum h_i^{(1)} \otimes h_i^{(2)}
$$
and $\cdot$ is the multiplication in $H$.

\begin{definition}
{\it Hopf algebra} is a bialgebra with a unit,
a counit and an antipode.
\end{definition}

\begin{example}
1) Let $G$ be a group, $K[G]$ be a group algebra. Let as define comultiplication $\Delta$, counit $\varepsilon$ and antipode $S$ on elements of $G$:
$$
\Delta(g) = g \otimes g,~~~\varepsilon(g) = 1,~~~s(g)= g^{-1},
$$
and extend by linearity on $K[G]$. We get a cocommutative Hopf algebra.

2) Let $G$ be a finite group. The Hopf algebra $K[G]^*$ has a basis $P_g$, $g \in G$ on which comultiplication and multiplication are defined by
$$
\Delta(P_g) = \sum_h P_h \otimes P_{h^{-1} g},~~~P_g P_h = \delta_{g, h} P_g.
$$
It means that $\{ P_g~|~g \in G \}$ is a set of pairwise orthogonal idempotents sum of which is equal to unit.
Further, counit is defined by equalities
$$
\varepsilon(P_1) = 1,~~~\varepsilon(P_g) = 0,~~g \in G, g \not= 1,
$$
and the antipode is defined by the equality
$$
S(P_g) = P_{g^{-1}}
$$
\end{example}

\subsection{Extensions of Yang-Baxter sets}

Let $X$ be a non-empty set and
$$
R  : X \times X \to X \times X
$$
be a  solution for the Yang-Baxter equation,
$$
R_{12} R_{13} R_{23} = R_{23} R_{13} R_{12}.
$$
(We call such a pair $(X,R)$ a Yang-Baxter set.)
We denote the components of $R$ as $R(x, y) = (\sigma_y(x), \tau_x(y))$ for $x, y \in X$.
If $(X, R^X)$ and $(Y, R^Y)$ are two Yang-Baxter sets then a map $f : X \to Y$ is said to be a {\it morphism} if the following diagram is commutative
$$
 \begin{diagram}
\node{X \times X}
\arrow[2]{e,t}{f\times f}
\arrow{s,l}{R^X}
\node[2]{Y \times Y}\arrow{s,r}{R^Y} \\
\node{X \times X}
\arrow[2]{e,b}{f\times f}
\node[2]{Y \times Y}
\end{diagram}
$$
i.e. for any $x, x' \in X$ holds
$$
R^Y (f \times f) (x, x') = (f \times f) R^X (x, x').
$$
For any $y \in Y$ we can defined its preimage
$$
f^{-1} (y) = \{ x \in X~|~f(x) = y\}.
$$
We will say that $f$ is {\it homogeneous} if cardinalities of all preimages $f^{-1} (y)$ are equal. In this case we can find a set of different elements $y_i$, $i \in I$, such that
$X$ is the disjoint union
$$
X = \coprod_{i \in I} f^{-1} (y_i).
$$
\bigskip
If for some $k \in I$ the following holds
$$
R^X (f^{-1} (y_k) \times f^{-1} (y_k)) \subseteq f^{-1} (y_k) \times f^{-1} (y_k),
$$
then we will say that there is a homomorphism of the solution $(X, R^X)$ to the solution $((Y, y_k), R^Y)$ with the kernel $(f^{-1} (y_k), R_{f^{-1} (y_k)})$.

In some problems it arises a different equivalence relation between the solutions for the Yang-Baxter equation. In the braided set case it was introduced a so-called guitar map 
which transfroms a solution to YBE on~$X$ to a solution of the special kind:
$$
R'(x,y) = (\sigma_y(\tau_{\sigma^{-1}_x(y)}(x)), y).
$$
This transformation was introduced by Soloviev \cite{Sol} and developed by Lebed and Vendramin in \cite{LV}.

\bigskip

\subsection{Extension of braided sets induced by a group structure}
\label{groupext}
Let us recall some ideas and results from \cite{PT}. A solution $S$ for the braid equation
\bea
\label{braid relation}
S_{12} S_{23} S_{12} =S_{23} S_{12} S_{23}
\eea
can be associated with the following solution for the Yang-Baxter equation
\bea
R(x,y)=P S (x,y).\nn
\eea
Let us recall  that $X$ is called a braided set if $X$ is equipped with $S : X^2\rightarrow X^2$ - a solution for the braid equation (\ref{braid relation}). 

\begin{definition}
Let us call a set $X$ with a binary algebraic operation $\triangleleft$ self-distributive if $\triangleleft$ satisfies
$$
(x\triangleleft y)\triangleleft z=(x\triangleleft z)\triangleleft(y\triangleleft z).
$$
\end{definition}

\begin{proposition} \label{prop_distr}
The set $(X,\triangleleft)$ is self-distributive $\Leftrightarrow$ the map $$S_\triangleleft(x,y)\stackrel{\rm def}{=}(y,x\triangleleft y)$$
defines a braided set on $X$.
\end{proposition}

\begin{example}
\label{ex_grdistr}
Any group $G$ with the conjugation operation $x\triangleleft y = y^{-1}xy$ is a self-distributive set. We call such self-distribute sets grouplike.
\end{example}

These observations allow to connect groups with braided sets. In particular this allows to describe extensions of grouplike braided sets. This is principally due to the well developed theory of group extensions defined by group cohomology \cite{Braun}. This idea is exploited in \cite{PT} to describe solutions for the parametric Yang-Baxter equation.

\subsection{Braid group and virtual pure braid group}

The braid group $B_n$ on $n$ strands $n \geq 2$ is generated by elements  $\sigma_1$,  $\sigma_2$, $\ldots$, $\sigma_{n-1}$. The relations of $B_n$ are given by
\bea
\sigma_i \sigma_{i+1} \sigma_i&=&\sigma_{i+1} \sigma_i \sigma_{i+1};\nn\\
\sigma_i \sigma_j&=& \sigma_j \sigma_i\qquad \mbox{if}\qquad |i-j|>1.\nn
\eea
A generator $\sigma_i$ has a geometric interpretation as the braidings of the $i$-th and $(i+1)$-th strands.

The {\it virtual braid group} $VB_n$  introduced in \cite{K}  is generated by the  braid group $B_n = \langle \sigma_1, \sigma_2, \ldots, \sigma_{n-1} \rangle$ and the symmetric group $S_n=\langle \rho_1, \rho_2,\ldots, \rho_{n-1} \rangle$  with the following relations:
\begin{align*}
\sigma_i \sigma_{i+1} \sigma_i&=\sigma_{i+1} \sigma_i \sigma_{i+1} & i=1, 2, \ldots, {n-2}, \\
\sigma_i \sigma_j&=\sigma_j \sigma_i &  |i-j| \geq 2, \\
\rho_i^{2}&=1 &  i=1, 2, \ldots, {n-1},\\
\rho_i \rho_j&= \rho_j \rho_i &  |i-j| \geq 2,\\
\rho_i \rho_{i+1} \rho_i&= \rho_i \rho_{i+1} \rho_i & i=1, 2, \ldots, {n-2},\\
\sigma_i \rho_j&= \rho_j \sigma_i & |i-j| \geq 2 ,\\
\rho_i \rho_{i+1} \sigma_i&= \sigma_{i+1} \rho_i \rho_{i+1} & i=1, 2, \ldots, {n-2}.
\end{align*}

The virtual pure braid group $VP_n$, $n\geq 2$, was introduced in  \cite{B} as the kernel of the homomorphism $VB_n \to S_n$, $\sigma_i \mapsto (i, i+1)$, 
$\rho_i \mapsto (i, i+1)$ for all $i=1, 2, \ldots, i-1$. $VP_n$
 admits a
presentation with the  generators $\lambda_{ij},\ 1\leq i\neq j\leq n,$
and the following relations:
\begin{align}
& \lambda_{ij}\lambda_{kl}=\lambda_{kl}\lambda_{ij} \label{rel},\\
&
\lambda_{ki}\lambda_{kj}\lambda_{ij}=\lambda_{ij}\lambda_{kj}\lambda_{ki}
\label{relation},
\end{align}
where distinct letters stand for distinct indices.

These generators of $VP_n$ can be expressed from the generators of $VB_n$ by the formulas 
\begin{align*}
\lambda_{i,i+1} &= \rho_i \, \sigma_i,\\
  \lambda_{i+1,i} &= \rho_i \, \lambda_{i,i+1} \, \rho_i = \sigma_i \, \rho_i
\end{align*}
for  $i=1, 2, \ldots, n-1$, and
\begin{align*}
\lambda_{i,j} & = \rho_{j-1} \, \rho_{j-2} \ldots \rho_{i+1} \, \lambda_{i,i+1} \, \rho_{i+1} \ldots \rho_{j-2} \, \rho_{j-1}, \\ 
\lambda_{j,i} & =  \rho_{j-1} \, \rho_{j-2} \ldots \rho_{i+1} \, \lambda_{i+1,i} \, \rho_{i+1} \ldots \rho_{j-2} \, \rho_{j-1}
\end{align*} 
for $1 \leq i < j-1 \leq n-1$.

We have decomposition $VB_n = VP_n \rtimes S_n$ and $S_n$ acts on $VP_n$ by the following rules:

\begin{lemma}[\cite{B}] \label{form}
Let $a$ be an element of $\langle \rho_1, \rho_2, \ldots, \rho_{n-1} \rangle$ and $\bar{a}$ is its image in $S_n$ under the isomorphism $\rho_i \mapsto (i,i+1)$, $i = 1, 2, \ldots, n-1$, then for any generator $\lambda_{ij}$ of $VP_n$ the following holds
$$
a^{-1} \lambda_{ij} a = \lambda_{(i)\bar{a}, (j)\bar{a}},
$$
where $(k)\bar{a}$ is the image of $k$ under the action of the permutation $\bar{a}$.
\end{lemma}

The group $VB_n$ contains a subgroup $VP_n^+$ that is generated by elements $\lambda_{ij},\ 1\leq i < j\leq n,$ moreover the map $\lambda_{ji} \mapsto \lambda_{ij}^{-1}$, $1\leq i < j\leq n,$ defines a homomorphism of $VP_n$ to $VP_n^+$. 

\begin{example} $VP_3$ is generated by elements
$$
\lambda_{12},~~\lambda_{21},~~\lambda_{13},~~\lambda_{23},~~\lambda_{31},~~\lambda_{32},
$$
and is defined by the relations
$$
\lambda_{ki}\lambda_{kj}\lambda_{ij}=\lambda_{ij}\lambda_{kj}\lambda_{ki}
$$
for distinct $i, j, k \in \{1, 2, 3 \}.$
It contains 
$$
VP_3^+ = \langle \lambda_{12}, \lambda_{13}, \lambda_{23}~|~  \lambda_{12} \lambda_{13} \lambda_{23} = \lambda_{23} \lambda_{13}  \lambda_{12} \rangle
$$
that is a group with one defining relation.
\end{example}

\medskip

It is possible to define another epimorphism $\mu: VB_n \to S_n$ as follows:
$$
\mu(\sigma_i)=1, \; \mu(\rho_i)=\rho_i, \;i=1,2,\dots, n-1\, ,
$$
where $S_n$ is generated by $\rho_i$ for $i=1,2,\dots, n-1$.
Let us denote by $H_n$ the normal closure of $B_n$ in $VB_n$.
It is evident that $\ker \mu$ coincides with $H_n$.
Let us define elements:
$$
x_{i,i+1}=\sigma_i,~~x_{i+1,i}=\rho_i \sigma_i \rho_i =\rho_i x_{i,i+1} \rho_i,
$$
for $i= 1, 2, \ldots, n-1$, and 
$$
x_{i,j}=\rho_{j-1} \cdots \rho_{i+1} \sigma_i \rho_{i+1} \cdots \rho_{j-1},
$$
$$
x_{j,i}=\rho_{j-1} \cdots \rho_{i+1} \rho_i \sigma_i \rho_i \rho_{i+1} \cdots \rho_{j-1},
$$
for $1 \le i < j-1 \le n-1$.

The group $H_n$ admits a presentation with the generators $x_{k,\, l},$ $1 \leq k \neq l \leq
n$,
and the defining relations:
\begin{equation} \label{eq40}
x_{i,j} \,  x_{k,\, l} = x_{k,\, l}  \, x_{i,j},
\end{equation}
\begin{equation} \label{eq41}
x_{i,k} \,  x_{k,j} \,  x_{i,k} =  x_{k,j} \,  x_{i,k} \, x_{k,j},
\end{equation}
where  distinct letters stand for distinct indices.

This presentation was found in \cite{R} (see also \cite{BB}). There is a decomposition
$VB_n = H_n \rtimes S_n$, where  $S_n=\langle \rho_1, \dots, \rho_{n-1} \rangle$ acts on the set
$\{x_{i,j} \, , \;  1 \le i \not= j \le n \}$ by permutation of indices and we have the next analogous of Lemma \ref{form}.

\begin{lemma}[\cite{BB}] \label{form1}
Let $a$ be an element of $\langle \rho_1, \rho_2, \ldots, \rho_{n-1} \rangle$ and $\bar{a}$ is its image in the symmetric group $S_n$ under the isomorphism $\rho_i \mapsto (i,i+1)$, $i = 1, 2, \ldots, n-1$, then for any generator $x_{ij}$ of $H_n$ the following holds
$$
a^{-1} x_{ij} a = x_{(i)\bar{a}, (j)\bar{a}},
$$
where $(k)\bar{a}$ is the image of $k$ under the action of the permutation $\bar{a}$.
\end{lemma}

\bigskip

\section{Braided sets and representations  of $VB_n$}
\label{sec-Br}

\subsection{Solutions of the YBE and representations of $VP_n$}
As above  $P(x, y) = (y, x)$ is a transposition in  $X^2 \to X^2$. For $n>2$ define the maps $P_{ij}: X^n \to X^n$, $1 \leq i < j \leq n$, which act as $P$ on $i$-th and $j$-th factors. It is easy to see that the group generated by $P_{i,i+1}$, $i=1, 2, \ldots, n-1$, is isomorphic to symmetric group $S_n$ and there is an isomorphism of $\langle \rho_1, \rho_2, \ldots, \rho_{n-1} \rangle$ to $\langle P_{12}, P_{23}, \ldots, P_{n-1,n} \rangle$, which is defined by the rule $\rho_i \mapsto P_{i,i+1}$.

Suppose that $(X, R)$ is a Yang--Baxter set with invertible map $R : X^2 \to X^2$. For any natural $n \ge 2$ we define the maps $R_{ij} : X^n \to X^n$ by the formulas
$$
R_{i,i+1} = \id^{i-1} \times R \times \id^{n-i-1},~~R_{i+1,i} = P_{i,i+1} R_{i,i+1} P_{i,i+1}
$$
for  $i=1, 2, \ldots, n-1$, and
\begin{align*}
R_{i,j} & = P_{j-1,j} \, P_{j-2,j-1} \ldots P_{i+1,i+2} \, R_{i,i+1} \, P_{i+1,i+2} \ldots P_{j-2,j-1} \, P_{j-1,j}, \\ 
R_{j,i} & = P_{j-1,j} \, P_{j-2,j-1} \ldots P_{i+1,i+2} \, R_{i+1,i} \,  P_{i+1,i+2} \ldots P_{j-2,j-1} \, P_{j-1,j},
\end{align*} 
for $1 \leq i < j-1 \leq n-1$.

\begin{proposition} \label{virt}
Suppose that $(X, R)$ is a Yang--Baxter set with invertible $R: X \times X \to X \times X$. Then there is a representation
$$
\varphi_n^{R}: VP_n \to \Sym(X^{n}),~~\mbox{for any}~n \geq 2.
$$
\end{proposition}

\begin{proof}
We proceed by induction on $n$. For $n=2$ the group $VP_2 = \langle \lambda_{12}, \lambda_{21} \rangle$ is isomorphic to $F_2$. The map $\lambda_{12} \to R_{12}$, $\lambda_{21} \to R_{21}$ defines a representation 
$$
\varphi_2^{R}: VP_2 \to \Sym(X^{2}). 
$$
For $n>2$ define the map  $\lambda_{kl} \mapsto R_{kl}$, $1 \leq k\not=l \leq n$. To prove that this map  defines the expected representation
$$
\varphi_n^{R}: VP_n \to \Sym(X^{n}),
$$
it is enough to show that in $\Sym(X^{n})$  the following relations hold
\begin{align}
& R_{ij} \, R_{kl}=R_{kl} \, R_{ij} \label{rel1},\\
&
R_{ki} \, R_{kj} \, R_{ij}=R_{ij} \, R_{kj} \, R_{ki}
\label{relation1},
\end{align}
where distinct letters stand for distinct indices.

By the construction, the map $R_{ij}$ acts as $R$ on the $i$-th and $j$-th factors and as identity on others. Hence the commutativity relations hold.
On the other side since $(X, R)$ is a Yang--Baxter set we get
$$
R_{12} \, R_{13} \, R_{23}=R_{23} \, R_{13} \, R_{12}.
$$
 From Lifting theorem \cite{BW} it follows that all relations 
 (\ref{relation1}) hold for arbitrary $n$.
\end{proof}

\begin{remark}
To prove that the relation 
$$
R_{12} \, R_{13} \, R_{23}=R_{23} \, R_{13} \, R_{12}
$$
induces all relations  (\ref{relation1}) instead Lifting theorem we use  conjugations by elements $P_{ij}$ for which the analog of Lemma \ref{form} holds. 
\end{remark}

\subsection{Solutions of the BE and representations of $H_n$}
It is well known that an invertible solution $(X, S)$ of the BE gives a representation $B_n \to \Sym(X^n)$. We prove some generalization of this fact. To do this let us introduce 
the maps $S_{ij} : X^n \to X^n$, $1 \le i \not= j \le n$ by the rules
\begin{align*}
S_{i,i+1} &=  S_i,\\
S_{i+1,i} &= P_{i,i+1} \, S_{i,i+1} \, P_{i,i+1},
\end{align*}
for  $i=1, 2, \ldots, n-1$, and
\begin{align*}
S_{i,j} & = P_{j-1,j} \, P_{j-2,j-1} \ldots P_{i+1,i+2} \, S_{i,i+1} \, P_{i+1,i+2} \ldots P_{j-2,j-1} \, P_{j-1,j}, \\ 
S_{j,i} & = P_{j-1,j} \, P_{j-2,j-1} \ldots P_{i+1,i+2} \, S_{i+1,i} \,  P_{i+1,i+2} \ldots P_{j-2,j-1} \, P_{j-1,j},
\end{align*} 
for $1 \leq i < j-1 \leq n-1$.

The proof of the next proposition is the same as the proof of Proposition \ref{virt}

\begin{proposition}
Suppose that $(X, S)$ is a braided set with invertible $S: X \times X \to X \times X$. Then the map $x_{\ij} \to S_{ij}$ defines a representation
$$
\varphi_n^{S}: H_n \to \Sym(X^{n}),~~\mbox{for any}~n \geq 2.
$$
\end{proposition}

\begin{corollary}
Suppose that $(X, S)$  is a braided set with invertible $S: X \times X \to X \times X$, or $(X, R)$  is a Yang-Baxter set with invertible $R: X \times X \to X \times X$.
Then there is a representation
$$
\psi_n: VB_n \to \Sym(X^{n}),~~\mbox{for any}~n \geq 2,
$$
which is an extension of the representation $\varphi_n^{S}$ and $\varphi_n^{R}$.
\end{corollary}

\begin{proof}
If we have a braided set $(X, S)$ with invertible $S$, then define $\psi_n$ on the generators of $VB_n$ by the rules
$$
\psi_n(\sigma_i) = S_{i, i+1},~~\psi_n(\rho_i) = P_{i, i+1},~~i = 1, 2, \ldots, n-1.
$$
As we know, $\psi_n$ define  representations of $B_n$ and $S_n$. We have to check that the mixed relations of $VB_n$ hold after applying $\psi_n$. The relation
$$
\sigma_i \rho_j = \rho_j \sigma_i,~~ |i-j| \geq 2 ,
$$
goes to the relation 
$$
S_{i,i+1} P_{j,j+1} = P_{j,j+1} S_{i,i+1},~~ |i-j| \geq 2,
$$
which evidently holds in $\Sym(X^n)$.
The second mixed relation
$$
\rho_i \rho_{i+1} \sigma_i = \sigma_{i+1} \rho_i \rho_{i+1},   ~~ i=1, 2, \ldots, {n-2}.
$$
goes to the relation
$$
P_{i,i+1} P_{i+1,i+2} S_{i,i+1} = S_{i+1,i+2} P_{i,i+1} P_{i+1,i+2},   ~~ i=1, 2, \ldots, {n-2},
$$
which is equivalent to 
$$
S_{i,i+1}^{ P_{i+1,i+2}} = S_{i+1,i+2}^{P_{i,i+1}}.
$$
Using the conjugation rules we get $S_{i, i+2} = S_{i, i+2}$.

As we know $VB_n = H_n \rtimes S_n$ and the restriction of $\psi_n$ to $H_n$ gives the representation $\varphi_n^S : H_n \to \Sym(X^n)$.

Let us recall that for each Yang--Baxter set $(X, R)$ with invertible $R$ we can construct a braided set $(X, S)$ where $S = P R$ is also invertible.
\end{proof}

\begin{remark}
For each invertible linear solution $(V, S)$ of the braid equation or invertible linear solution $(V, R)$ of the YBE we can construct a representation $VB_n \to \Aut(V^{\otimes n})$.
\end{remark}

\section{Extension of quasi-triangular Hopf algebras}
\label{sec-Ext}
Recall that a {\it quasi\-trian\-gular Hopf algebra} is a Hopf algebra $A$ together with an
invertible element $R \in A \otimes A$ (the quantum $R$-matrix) satisfying the `exchange'
condition
$$
\Delta^{op}(a) = R \Delta(a) R^{-1},~~a \in A,
$$
and the compatibility condition
$$
(\Delta\otimes \id)R=R_{23} R_{13},~~(\id\otimes \Delta)R=R_{12} R_{13}.
$$
These conditions imply that $R$ satisfies the quantum Yang--Baxter equation
$$
R_{12} R_{13} R_{23} = R_{23}  R_{13} R_{12}.
$$

The `classical analogue' of a quasitriangular Hopf algebra is realized  in the context of Lie algebras. A {\it Lie bialgebra} is a dual pair $(\mathfrak{g}, \mathfrak{g}^*)$ of Lie
algebras for which the dual of commutator on $\mathfrak{g}^*$, $c : \mathfrak{g} \to \mathfrak{g} \land \mathfrak{g}$, is a cocycle with respect to the adjoint representation. The Lie bialgebra is called {\it quasitriangular} if
it is equipped with an element $r \in  \mathfrak{g} \otimes \mathfrak{g}$ (the classical $r$-matrix) satisfying conditions which are `infinitesimal versions' of those for $R$, 
the most important of them
being the classical Yang--Baxter equation
$$
[r_{12}, r_{13}] + [r_{12}, r_{23}] + [r_{13}, r_{13}] = 0. 
$$
We consider
$\mathfrak{g}$ as embedded in $U(\mathcal{\mathfrak{g}})$.

\subsection{Product of Hopf algebras}

Recall some definitions and constructions which one can find in  \cite[Chapter 4]{CP}.

A Hopf algebra $A$ over a commutative ring $k$ is said to be {\it almost cocommutative} if there is an invertible element $R\in A\otimes A$ such that
$$
\Delta^{op}(a) = R  \Delta(a) R^{-1}
$$
for all $a \in A$.

Recall a construction of the product of Hopf algebras. Let $B(\Delta^B, \varepsilon^B, S^B, R^B)$  and $C(\Delta^C, \varepsilon^C, S^C, R^C)$ be Hopf algebras over a commutative ring $K$, 
and let $R\in C\otimes B$ be an invertible element such that
\bea
\label{spl}
(\Delta^C\otimes id)R=R_{23} R_{13},&\qquad& (id\otimes \Delta^B)R=R_{12} R_{13},\\
(id\otimes S^B)R=R^{-1},&\qquad&(S^C\otimes id)R=R^{-1}.\nn
\eea
Then the tensor product $B\otimes C$ can be endowed  with the  Hopf algebra structure with the multiplication of the tensor product, a comultiplication given by
\bea
\Delta(b\otimes c)=R_{23}\Delta_{13}^B(b) \Delta_{24}^C(c) R_{23}^{-1};\nn
\eea
an antipode
\bea
S(b\otimes c)=R_{21}^{-1} (S^B(b)\otimes S^C(c))R_{21}\nn
\eea
and a counit
\bea
\varepsilon(b\otimes c)=\varepsilon^B(b) \varepsilon^C(c).\nn
\eea
This algebra is denoted by $B \underset{R}{\otimes} C$.

\begin{theorem}
\label{th-1}
Let under the conditions of the previous construction the Hopf algebras $B,C$ be quasi-triangular, i.e. such that the comultiplication is quasi-cocommutative
\bea
(\Delta^B)^{op}&=&R^B \Delta^B (R^B)^{-1};\nn\\
(\Delta^C)^{op}&=&R^C \Delta^C (R^C)^{-1};\nn
\eea
and moreover the following conditions fulfill
\bea
(\Delta^B\otimes id)R^B=R_{13}^B R_{23}^B;&\qquad& (id \otimes \Delta^B)R^B=R_{13}^B R_{12}^B;\nn\\
(\Delta^C\otimes id)R^C=R_{13}^C R_{23}^C;&\qquad& (id \otimes \Delta^C)R^C=R_{13}^C R_{12}^C.\nn
\eea
In such a case the Hopf algebra $B\otimes C$ is also a quasi-triangular with the structure element (the quantum $R$-matrix)
\bea
\mathcal{R}=R_{41}R_{13}^B R_{24}^C R_{23}^{-1}.\nn
\eea
\end{theorem}
A part of this theorem could be found in \cite{RS}.

\begin{proof}
Let us represent the opposite comultiplication on $B\otimes C$ in the form:
\bea
\Delta^{op}(b\otimes c)&=&P_{13} P_{24} \Delta(b\otimes c)=R_{41} (\Delta_{13}^{B})^{op}(b) (\Delta_{24}^{C})^{op}(c) R_{41}^{-1}\nn\\
&=&R_{41}R_{13}^B \Delta_{13}^B(b) (R_{13}^B)^{-1} R_{24}^C \Delta_{24}^C(c) (R_{24}^C)^{-1} R_{41}^{-1}\nn\\
&=&R_{41}R_{13}^B R_{24}^C R_{23}^{-1} \Delta(b\otimes c)  R_{23}(R_{13}^B)^{-1} (R_{24}^C)^{-1} R_{41}^{-1}.\nn
\eea

Thus:
\bea
\Delta^{op}(b\otimes c)=\mathcal{R}\Delta(b\otimes c) \mathcal{R}^{-1}.\nn
\eea
We now prove that this Hopf algebra is quasi-triangular.
We derive several consequences of the conditions (\ref{spl}) on $R$
\bea
\label{eq2:comp1}
R_{12}^C R_{23} R_{13} &=&R_{13} R_{23} R_{12}^C,\\
\label{eq2:comp2}
R_{23}^B R_{12} R_{13}&=&R_{13} R_{12} R_{23}^B.
\eea
We will use multi-indexes to denote internal tensor components. For example, $\mathcal{R}=R_{41}R_{13}^B R_{24}^C R_{23}^{-1}$ is indexed as follows:
\bea
\mathcal{R}_{(12)(34)}.\nn
\eea
This underlines that this is an element of the space $(B\otimes C)\otimes (B\otimes C).$
Let us prove that:
\bea
(\Delta_{(12)}\otimes id)\mathcal{R}=\mathcal{R}_{(12)(56)} \mathcal{R}_{(34)(56)}.\nn
\eea
The R.H.S. of this expression takes the form:
\bea
\label{eq1:rhs}
R_{61} R_{15}^B R_{26}^C R_{25}^{-1} R_{63} R_{35}^B R_{46}^C R_{45}^{-1}.
\eea
We will calculate the left part sequentially. First we will apply the operation$$P_{23}(\Delta^B\otimes \Delta^C\otimes id \otimes id)$$ to $\mathcal{R}_{(12)(34)}.$ We obtain:
\bea
R_{61} R_{63} R_{15}^B R_{35}^B R_{26}^C R_{46}^C R_{25}^{-1} R_{45}^{-1}.
\eea
Then we conjugate an expression with $R_{23}$.
\bea
R_{23}R_{61} R_{63} R_{15}^B R_{35}^B R_{26}^C R_{46}^C R_{25}^{-1} R_{45}^{-1} R_{23}^{-1}&=&
R_{61}{\color{blue} R_{23} R_{63}  R_{26}^C} {\color{magenta}R_{15}^B} R_{35}^B  {\color{magenta} R_{46}^C} R_{25}^{-1} R_{45}^{-1} R_{23}^{-1}
\nn\\
R_{61} {\color{magenta}R_{15}^B}  {\color{blue} R_{26}^C R_{63}  R_{23}}  R_{35}^B   R_{25}^{-1} {\color{magenta} R_{46}^C} R_{45}^{-1} R_{23}^{-1}&=&R_{61} R_{15}^B   R_{26}^C R_{63} {\color{blue} R_{23}  R_{35}^B   R_{25}^{-1}}  R_{46}^C R_{45}^{-1} R_{23}^{-1}\nn\\
R_{61} R_{15}^B   R_{26}^C R_{63} {\color{blue} R_{23}  R_{35}^B   R_{25}^{-1}}  R_{46}^C R_{45}^{-1} {\color{magenta} R_{23}^{-1}}&=&R_{61} R_{15}^B   R_{26}^C R_{63} {\color{blue} R_{25}^{-1}   R_{35}^B   R_{23}}  {\color{magenta} R_{23}^{-1}} R_{46}^C R_{45}^{-1} .\nn
\eea
This coincides with (\ref{eq1:rhs}). Here we applied (\ref{eq2:comp1}) and (\ref{eq2:comp2}), the corresponding multipliers were highlighted in blue. In purple we highlighted multipliers that could be freely rearranged. Similarly we prove the identity:
\bea
(id \otimes \Delta_{(34)})\mathcal{R}=\mathcal{R}_{(12)(56)} \mathcal{R}_{(12)(34)}.\nn
\eea
\end{proof}

\subsection{Drinfeld twist}
Let $T\in GL(V\otimes V)$ satisfies the braid relation
\bea
T_{12} T_{23} T_{12}=T_{23} T_{12} T_{23}.\nn
\eea

Let also consider $F\in GL(V\otimes V)$ and $\Psi$ and $\Phi$ in $GL(V\otimes V\otimes V)$ such that
\bea
F_{12} \Psi&=&F_{23} \Phi\nn\\
\Phi T_{23}&=&T_{23} \Phi\nn\\
\Psi T_{12}&=&T_{12} \Psi \nn
\eea
then $R=F T F^{-1}$ also satisfies the braid relation
\bea
R_{12} R_{23} R_{12}=R_{23} R_{12} R_{23}.\nn
\eea
Such a transformation is called Drinfeld twist  \cite{DrTw}. This construction was originally proposed in the tensor category in the context of quasi-triangular Hopf algebras deformations. In fact the construction of Theorem \ref{th-1} can be considered as an version of the Drinfeld twist. Let us firstly shift to the braid notation with the help of transpositions $P^B$ and $P^C$ in $B\otimes B$ and $C\otimes C$ respectively:
\bea
R^B&=&P^B T^B;\nn\\
R^C&=&P^C T^C;\nn\\
\mathcal{R}_{(12)(34)}&=&P^B_{13}P^C_{24}\mathcal{T}_{(12)(34)}.\nn
\eea
Then
\bea
\mathcal{T}_{(12)(34)}=R_{23}T^B_{13} T^C_{24} R_{23}^{-1}.\nn
\eea
This expression is a conjugation of the trivial solution $T^B_{13} T^C_{24} $ for the braid relation on the tensor product $B\otimes C.$ 
Let us compare this transformation with the Drinfeld twist.  In terms of  $T^B$ and $T^C$ the equations (\ref{eq2:comp1}) and (\ref{eq2:comp2}) take the form:
\bea
T_{12}^C R_{23} R_{13}&=&R_{23} R_{13} T_{12}^C;\label{eq4:1}\\
T_{23}^B R_{12} R_{13}&=&R_{12} R_{13} T_{23}^B.\label{eq4:2}
\eea
To get exactly the Drinfeld twist we have to find such elements $\Psi$ and $\Phi$ in $(B\otimes C)^{\otimes ^3}$ that
\bea
F_{(12)(34)} \Psi_{(12)(34)(56)}&=&F_{(34)(56)}\Phi_{(12)(34)(56)};\label{eq3:1}\\
\Phi_{(12)(34)(56)} T_{35}^B T_{46}^C&=&T_{35}^B T_{46}^C \Phi_{(12)(34)(56)};\label{eq3:2}\\
\Psi_{(12)(34)(56)} T_{13}^B T_{24}^C&=&T_{13}^B T_{24}^C \Psi_{(12)(34)(56)}.\label{eq3:3}
\eea
Due to (\ref{eq4:1}) and (\ref{eq4:2}) the following choice
\bea
\Phi&=&R_{23} R_{25};\nn\\
\Psi&=&R_{45} R_{25};\nn
\eea
guarantees equation (\ref{eq3:2}), (\ref{eq3:3}) and (\ref{eq3:1}). This choice is argued by \cite{Kul}.

\subsection{Infinitesimal version in tensor case}
Here we deduce a classical limit of Theorem~\ref{th-1}.
\begin{theorem} \label{tc}
Let $r^B$ and $r^C$ are solutions for the classical Yang-Baxter equation (CYBE) on Lie algebras $B$ and $C$ respectively and $r$ satisfies equations:
\bea
\label{fused2}
[r_{12}^C,r_{13}+r_{23}]=[r_{23},r_{13}];
\eea
\bea
\label{fused}
[r_{12}^B,r_{12}+r_{13}]=[r_{12},r_{13}].
\eea
Then
\bea
\widetilde{r}_{(12)(34)}=r^B_{13}+r^C_{24}+r_{41}-r_{23}\nn
\eea
solves the CYBE on $B\otimes C.$
\end{theorem}
\begin{proof}
The CYBE in this case takes the form:
\bea
[\widetilde{r}_{(12)(34)},\widetilde{r}_{(12)(56)}+\widetilde{r}_{(34)(56)}]+[\widetilde{r}_{(12)(56)},\widetilde{r}_{(34)(56)}]=0\nn
\eea
This could be expressed as follows:
\bea
&&[r_{13}^B+r_{24}^C+r_{41}-r_{23},r_{15}^B+r_{26}^c+r_{61}-r_{25}+r_{35}^B+r_{46}^C+r_{63}-r_{45}]\nn\\
&+&[r_{15}^B+r_{26}^c+r_{61}-r_{25},r_{35}^B+r_{46}^C+r_{63}-r_{45}].\nn
\eea
All commutators of $r_{ij}^B$ with $r_{kl}^B$  provide CYBE for $r^B$
\bea
[r_{13}^B,r_{15}^B]+[r_{13}^B,r_{35}^B]+[r_{15}^B,r_{35}^B]=0.\nn
\eea
The same thing we see for the commutators of $r^C$ with $r^C.$
We observe that the commutators between $r^B$ and $r^C$ are all trivial due to their localization in different tensor component. Let us view in details the commutators of $r_{13}^B$ with different $r$'s.
\bea
[r_{13}^B,r_{61}+r_{63}]+[r_{61},r_{63}].\nn
\eea
This is zero due to (\ref{fused}). The similar formulas gather with other terms.
\end{proof}

\begin{remark}
The conditions (\ref{fused2}) and (\ref{fused}) on $r$ which are necessary for extension in infinitesimal case can be expressed as Maurer-Cartan condition in an appropriate differential graded Lie algebra. We suppose that it is a relevant version of the cohomological characterization of extensions in this case. We postpone the presentation of this material for the next publications.
\end{remark}

\bigskip

\section{Extensions in set-theoretic case}
\label{sec-Ext2}

\subsection{Product of Yang-Baxter sets}
In this section we are considering the next

\begin{question}
Let $B$ and $C$ are two non-empty sets and $R^B : B \times B \to B \times B$ and $R^C  : C \times C \to C \times C$ be Yang--Baxter maps (YBM) on them. It means that these maps satisfy the equalities
$$
R^{B}_{12} R^{B}_{13} R^{B}_{23} = R^{B}_{23} R^{B}_{13} R^{B}_{12},~~~R^{C}_{12} R^{C}_{13} R^{C}_{23} = R^{C}_{23} R^{C}_{13} R^{C}_{12}.
$$
What YBMs is it possible to define on the direct product $B \times C$?
\end{question}

There is an obvious way to take the direct product $R^B \times R^C$, which acts by the rule
$$
\left(R^B \times R^C \right) ((b_1, c_1), (b_2, c_2)) = \left( (\sigma_{b_2}^B(b_1), \sigma_{c_2}^C(c_1)), (\tau_{b_1}^B(b_2), \tau_{c_1}^C(c_2)) \right),
$$
where
$$
R^B(b_1, b_2) = \left( \sigma_{b_2}^B(b_1),  \tau_{b_1}^B(b_2) \right),~~~R^C(c_1, c_2) = \left( \sigma_{c_2}^C(c_1),  \tau_{c_1}^C(c_2) \right).
$$

\begin{remark}
If $B$ and $C$ are not only sets but some algebraic systems: groups, racks, bi-racks, skew braces, we can use extensions of these systems. The group case is emphasized in Subsection \ref{groupext}. Typically in these constructions the roles of the sets $B$ and $C$ are non-symmetric, one of them is the image and the second one is the kernel of a homomorphism. 
\end{remark}

Let us define a map $R :  C \times B \to C \times B$, $R(c, b) = (\mu(c, b), \nu(c, b))$ such that
$$
R^{B}_{23} R_{12} R_{13} = R_{13} R_{12} R^{B}_{23},~~~R^{C}_{12} R_{23} R_{13} = R_{13} R_{23} R^{C}_{12}.
$$
The both sides of the first equality are maps on $C \times B \times B \to C \times B \times B$, for the second one  -  on $C \times C \times B \to C \times C \times B$.

The next lemma is evident

\begin{lemma} \label{rel}
The next relations hold
\begin{align}
R_{15}^B R_{41} R_{45} = R_{45}  R_{41} R_{15}^B, \label{rb1}\\
 R_{45}^{-1} R_{15}^B R_{41} =  R_{41} R_{15}^B  R_{45}^{-1}, \label{rb2}\\
 R_{35}^B R_{23} R_{25} = R_{25}  R_{23} R_{35}^B, \label{rb3}\\
 R_{23}^{-1} R_{25}^{-1} R_{35}^B =  R_{35}^B R_{25}^{-1}  R_{23}^{-1}, \label{rb4}\\
R^B_{13}  R_{61} R_{63} =   R_{63} R_{61} R^B_{13}, \label{rb5}\\
R_{26}^C R_{63} R_{23} = R_{23} R_{63} R_{26}^C, \label{rc1}\\
 R_{23}^{-1} R_{26}^C R_{63} =  R_{63} R_{26}^C  R_{23}^{-1}, \label{rc2}\\
R_{24}^C R_{45} R_{25} = R_{25} R_{45} R_{24}^C, \label{rc3}\\
 R_{45}^{-1} R_{25}^{-1} R_{24}^C =  R_{24}^C R_{25}^{-1}  R_{45}^{-1}, \label{rc4}\\
 R_{46}^C R_{61} R_{41} = R_{41} R_{61} R_{46}^C \label{rc5}.
\end{align}
\end{lemma}

The next theorem is a set-theoretic analogous of Theorem \ref{th-1}.

\begin{theorem} \label{t}
The pair $(B \times C, \mathcal{R})$ where $\mathcal{R} = R_{41} R^B_{13} R^C_{24} R^{-1}_{23}$ is a  YB set.
\end{theorem}

\begin{proof}
We have to check the equality
$$
\mathcal{R}_{12} \mathcal{R}_{13}  \mathcal{R}_{23}  = \mathcal{R}_{23}  \mathcal{R}_{13} \mathcal{R}_{12}.
$$
Take its right side
$$
\mathcal{R}_{23}  \mathcal{R}_{13} \mathcal{R}_{12}  = R_{63} R^B_{35} R^C_{46} (R^{-1}_{45}  \cdot  R_{61}) R^B_{15} (R^C_{26} R^{-1}_{25} 
 \cdot R_{41}) R^B_{13} R^C_{24} R^{-1}_{23}.
$$
Using commutativity relations and Lemma \ref{rel} we will move $R_{41}$ to the left. Since $R_{41}$ commutes with $R^{-1}_{25}$ and $R^C_{26}$ and respectively  $R^{-1}_{45}$ commutes wih $R_{61}$ we obtain
$$
\mathcal{R}_{23}  \mathcal{R}_{13} \mathcal{R}_{12}  = R_{63} R^B_{35} R^C_{46} R_{61}  (R^{-1}_{45} R^B_{15} R_{41})  R^C_{26} R^{-1}_{25} 
 \cdot R^B_{13} R^C_{24} R^{-1}_{23}.
$$
By (\ref{rb2}),
$$
\mathcal{R}_{23}  \mathcal{R}_{13} \mathcal{R}_{12}  = R_{63} R^B_{35} (R^C_{46} R_{61} R_{41}) R^B_{15} R^{-1}_{45}   R^C_{26} R^{-1}_{25} 
 \cdot R^B_{13} R^C_{24} R^{-1}_{23}.
$$
By (\ref{rc5}),
$$
\mathcal{R}_{23}  \mathcal{R}_{13} \mathcal{R}_{12}  = R_{63} R^B_{35} (R^C_{46} R_{61} R_{41}) R^B_{15} R^{-1}_{45}   R^C_{26} R^{-1}_{25} 
 \cdot R^B_{13} R^C_{24} R^{-1}_{23}.
$$
Hence,
$$
\mathcal{R}_{23}  \mathcal{R}_{13} \mathcal{R}_{12}  = R_{63} R_{41} (R^B_{35}  R_{61}) R^C_{46}  R^B_{15} R^{-1}_{45}   R^C_{26} R^{-1}_{25} 
 \cdot R^B_{13} R^C_{24} R^{-1}_{23},
$$
and
$$
\mathcal{R}_{23}  \mathcal{R}_{13} \mathcal{R}_{12}  = (R_{63} R_{41}) R_{61} R^B_{35}  R^C_{46}  R^B_{15} R^{-1}_{45}   R^C_{26} R^{-1}_{25} 
 \cdot R^B_{13} R^C_{24} R^{-1}_{23},
$$
that is equal to
$$
\mathcal{R}_{23}  \mathcal{R}_{13} \mathcal{R}_{12}  = R_{41} R_{63} R_{61} R^B_{35}  R^C_{46}  R^B_{15} R^{-1}_{45}   R^C_{26} R^{-1}_{25} 
 \cdot R^B_{13} R^C_{24} R^{-1}_{23}.
$$

Take the  left hand side
$$
\mathcal{R}_{12} \mathcal{R}_{13}  \mathcal{R}_{23}  =  R_{41} R^B_{13} R^C_{24} (R^{-1}_{23} \cdot  R_{61} R^B_{15}) R^C_{26} (R^{-1}_{25} \cdot  R_{63}) R^B_{35} R^C_{46} R^{-1}_{45}.
$$
Using commutativity relations and Lemma \ref{rel} we will move $R_{63}$ to the left. Since $R^{-1}_{23}$ commutes with $R_{61} R^B_{15}$ and $R_{63}$ commutes 
with $R^{-1}_{25}$ we get
$$
\mathcal{R}_{12} \mathcal{R}_{13}  \mathcal{R}_{23}  =  R_{41} R^B_{13} R^C_{24}  R_{61} R^B_{15} (R^{-1}_{23} R^C_{26} R_{63}) R^{-1}_{25}  R^B_{35} R^C_{46} R^{-1}_{45}.
$$
By (\ref{rc2}), $R_{23}^{-1} R_{26}^C R_{63} =  R_{63} R_{26}^C  R_{23}^{-1}$. Hence
$$
\mathcal{R}_{12} \mathcal{R}_{13}  \mathcal{R}_{23}  =  R_{41} R^B_{13} R^C_{24}  R_{61} (R^B_{15} R_{63}) R_{26}^C  R_{23}^{-1} R^{-1}_{25}  R^B_{35} R^C_{46} R^{-1}_{45}.
$$
By commutativity we get
$$
\mathcal{R}_{12} \mathcal{R}_{13}  \mathcal{R}_{23}  =  R_{41} R^B_{13} (R^C_{24}  R_{61} R_{63}) R^B_{15} R_{26}^C  R_{23}^{-1} R^{-1}_{25}  R^B_{35} R^C_{46} R^{-1}_{45}.
$$
Since $R^C_{24}$ commutes with the product $ R_{61} R_{63}$ we obtain
$$
\mathcal{R}_{12} \mathcal{R}_{13}  \mathcal{R}_{23}  =  R_{41} (R^B_{13}  R_{61} R_{63}) R^C_{24}  R^B_{15} R_{26}^C  R_{23}^{-1} R^{-1}_{25}  R^B_{35} R^C_{46} R^{-1}_{45}.
$$
By (\ref{rb5}),
$$
\mathcal{R}_{12} \mathcal{R}_{13}  \mathcal{R}_{23}  =  R_{41}  R_{63} R_{61} R^B_{13}  R^C_{24}  R^B_{15} R_{26}^C  R_{23}^{-1} R^{-1}_{25}  R^B_{35} R^C_{46} R^{-1}_{45}.
$$
Comparing the right side of the last equality with the corresponding  equality for $\mathcal{R}_{23}  \mathcal{R}_{13} \mathcal{R}_{12}$ we see that we can reduce both sides from the left by $R_{41}  R_{63} R_{61}$. Hence, the equality
$$
\mathcal{R}_{12} \mathcal{R}_{13}  \mathcal{R}_{23} = \mathcal{R}_{23}  \mathcal{R}_{13} \mathcal{R}_{12}
$$  
takes the form
$$
R^B_{13}  R^C_{24}  R^B_{15} R_{26}^C  (R_{23}^{-1} R^{-1}_{25}  R^B_{35}) R^C_{46} R^{-1}_{45} = 
R^B_{35}  R^C_{46}  R^B_{15} (R^{-1}_{45}   R^C_{26}) R^{-1}_{25} R^B_{13} R^C_{24} R^{-1}_{23}.
$$
By (\ref{rb4}) and commutativity we get
$$
R^B_{13}  R^C_{24}  R^B_{15} R_{26}^C  R^B_{35} (R^{-1}_{25} R_{23}^{-1} R^C_{46}) R^{-1}_{45} = 
R^B_{35}  R^C_{46}  R^B_{15}  R^C_{26} (R^{-1}_{45}   R^{-1}_{25} R^B_{13}) R^C_{24} R^{-1}_{23}.
$$
By commutativity we induce
$$
R^B_{13}  R^C_{24}  R^B_{15} R_{26}^C  R^B_{35} R^C_{46} R^{-1}_{25} R_{23}^{-1} R^{-1}_{45} = 
R^B_{35}  R^C_{46}  R^B_{15}  R^C_{26} R^B_{13} (R^{-1}_{45}   R^{-1}_{25}  R^C_{24}) R^{-1}_{23}.
$$
 (\ref{rc4}) gives
$$
R^B_{13}  R^C_{24}  R^B_{15} R_{26}^C  R^B_{35} R^C_{46} R^{-1}_{25} R_{23}^{-1} R^{-1}_{45} = 
R^B_{35}  R^C_{46}  R^B_{15}  R^C_{26} R^B_{13}  R^C_{24} R^{-1}_{25} R^{-1}_{45}   R^{-1}_{23}.
$$
Since $R^{-1}_{45}$ commutes with $  R^{-1}_{23}$, we can reduce both sides of the last equality by $R^{-1}_{25} R_{23}^{-1} R^{-1}_{45}$ from the right,
$$
R^B_{13}  R^C_{24}  R^B_{15} R_{26}^C  R^B_{35} R^C_{46}  = 
R^B_{35}  R^C_{46}  R^B_{15}  R^C_{26} R^B_{13}  R^C_{24}.
$$
The last equality is equivalent to 
$$
R^B_{13}  R^B_{15}  R^B_{35} \cdot  R^C_{24} R_{26}^C  R^C_{46}  = 
R^B_{35} R^B_{15} R^B_{13} \cdot  R^C_{46}   R^C_{26}   R^C_{24},
$$
which is true.
\end{proof}


\subsection{Set-theoretic Drinfeld twist}

Suppose that $(X, T)$ is a set-theoretic solution to the  braid equation (BE) and there exist $F \in Sym(X \times X)$  and $\Phi, \Psi \in Sym(X\times X\times X)$ such that
\bea
F_{12} \Psi&=&F_{23} \Phi\nn,\\
\Phi T_{23}&=&T_{23} \Phi\nn,\\
\Psi T_{12}&=&T_{12} \Psi \nn,
\eea
then $R=F T F^{-1}$ also satisfies the braid equation
\bea
R_{12} R_{23} R_{12}=R_{23} R_{12} R_{23}.\nn
\eea

Using the notations of the previous section we put
$$
S^B = P R^B,~~S^C = P R^C.
$$
Then the relations 
$$
R_{12}^B R_{13}^B  R_{23}^B  = R_{23}^B  R_{13}^B  R_{12}^B,~~~R_{12}^C R_{13}^C  R_{23}^C  = R_{23}^C  R_{13}^C  R_{12}^C
$$
ensure the braid relations
$$
S_{12}^B S_{23}^B  S_{12}^B  = S_{23}^B  S_{12}^B  S_{23}^B,~~~S_{12}^C S_{23}^C S_{12}^C  = S_{23}^C  S_{12}^C  S_{23}^C.
$$
The relations
$$
R_{12}^C R_{23}  R_{13}  = R_{13}  R_{23}  R_{12}^C,~~~R_{23}^B R_{12}  R_{13}  = R_{13}  R_{12}  R_{23}^B,
$$
induce the following
$$
S_{12}^C R_{23} R_{13}  = R_{23}  R_{13}  S_{12}^C,~~~S_{23}^B R_{12} R_{13}  = R_{12}  R_{13}  S_{23}^B.
$$
We get the $R$-matrix
$$
\mathcal{R} = R_{41} R^B_{13} R^C_{24} R^{-1}_{23} = P_{12} P_{24} (R_{23} S^B_{13} S^C_{24} R^{-1}_{23}),
$$
and the elements
$$
\mathcal{R}_{12} =  P_{12} P_{24} (R_{23} S^B_{13} S^C_{24} R^{-1}_{23}),
$$
$$
\mathcal{R}_{13} =  P_{15} P_{26} (R_{25} S^B_{15} S^C_{26} R^{-1}_{25}),
$$
$$
\mathcal{R}_{23} =  P_{35} P_{46} (R_{45} S^B_{35} S^C_{46} R^{-1}_{45}).
$$
It is not difficult to prove

\begin{lemma}
The relation
$$
\mathcal{R}_{12} \mathcal{R}_{13}  \mathcal{R}_{23}  = \mathcal{R}_{23}  \mathcal{R}_{13}  \mathcal{R}_{12}
$$
provides the following
$$
\mathcal{S}_{12} \mathcal{S}_{23} \mathcal{S}_{12} = \mathcal{S}_{23} \mathcal{S}_{12} \mathcal{S}_{23},
$$
where
$$
\mathcal{S}_{12} = R_{23} S^B_{13} S^C_{24} R^{-1}_{23},~~\mathcal{S}_{23} = R_{45} S^B_{35} S^C_{46} R^{-1}_{45}.
$$
\end{lemma}

Hence, we have shown that if we take $T = S^B S^C$, which evidently satisfies the BE, then $\mathcal{S}=F T F^{-1}$ also satisfies the BE, where $F = F_{12} = R_{23}$ and $F_{23} = R_{45}$. 

Let us take
$$
\Phi =  R_{23} R_{25}, ~~~\Psi = R_{45} R_{25}.
$$
Then the first equality of the system has the form
$$
R_{23} \cdot (R_{45} R_{25}) = R_{45} \cdot (R_{23} R_{25}).
$$
This relation holds due to commutativity of $R_{23}$ and $R_{45}.$ 

From Lemma \ref{rel} it follows

\begin{lemma} \label{rel1}
The next relations hold
\begin{align}
S_{15}^B R_{41} R_{45} = R_{41}  R_{45} S_{15}^B, \label{brb1}\\
S_{35}^B R_{23} R_{25} = R_{23}  R_{25} S_{35}^B, \label{brb2}\\
S^B_{13}  R_{61} R_{63} =   R_{61} R_{63} S^B_{13}, \label{brb3}\\
S_{26}^C R_{63} R_{23} = R_{63} R_{23} S_{26}^C, \label{brc1}\\
S_{24}^C R_{45} R_{25} = R_{45} R_{25} S_{24}^C, \label{brc2}\\
S_{46}^C R_{61} R_{41} = R_{61} R_{41} S_{46}^C \label{brc3}.
\end{align}
\end{lemma}

The second equality of the system has the form
$$
( R_{23} R_{25}) (S_{35}^B S_{46}^C) = (S_{35}^B S_{46}^C) ( R_{23} R_{25}).
$$
Using (\ref{brb2}) and commutativity we rewrite the left hand side as
$$
 (R_{23} R_{25} S_{35}^B) S_{46}^C = S_{35}^B  R_{23} R_{25} S_{46}^C 
= S_{35}^B S_{46}^C  R_{23} R_{25}. 
$$
The third equality of the system has the form
$$
(R_{45} R_{25}) (S_{13}^B S_{24}^C) = (S_{13}^B S_{24}^C) (R_{45} R_{25}).
$$
To check  this equality we use the similar arguments as above.

\bigskip

\section{Groups  $Y_n$}
\label{sec-Gr}
In the previous two sections we constructed linear solution of the YBE on a tensor product $B \otimes C$ of two quasi-triangular  Hopf algebras $B$ and $C$ and a 
set-theoretic solution on the direct product $B \times C$ of two YB sets $B$ and $C$. There is a natural question: Is there any group such that these solutions are representations of this group. This section is devoted to this question.

\subsection{Definition}
We define groups  $Y_n$ for $n\geq 2$. Let us  start with the first non-trivial example $Y_3$. This group is generated by three families of elements:
$$
b_{13}, ~~b_{15}, ~~b_{35}, ~~c_{24},~~c_{26},~~c_{46},~~d_{23},~~d_{25}, ~~d_{41},~~d_{45}, ~~d_{61},~~d_{63},
$$
and is defined by relations:

-- {\it Yang-Baxter relations}
$$
b_{13} b_{15} b_{35} =  b_{35} b_{15} b_{13}, ~~~c_{24} c_{26} c_{46} =  c_{46} c_{26} c_{24};
$$

-- {\it commutativity relations}
$$
b_{ij} c_{kl} = c_{kl} b_{ij},~~b_{ij} \in \{ b_{13}, b_{15},  b_{35} \}, ~~c_{kl} \in \{ c_{24}, c_{26},  c_{46}\};
$$
$$
b_{ij} d_{kl} = d_{kl} b_{ij},~~c_{ij} d_{kl} = d_{kl} c_{ij},~d_{ij} d_{kl} = d_{kl} d_{ij},~ \mbox{where all indices}~ i, j, k, l~\mbox{are different};
$$

-- {\it mixed   relations}
$$
b_{13} d_{61} d_{63} =  d_{63}  d_{61} b_{13},~~b_{15} d_{41} d_{45} =  d_{45}  d_{41} b_{15},~~b_{35} d_{23} d_{25} =  d_{25}  d_{23} b_{35},
$$
$$
c_{24} d_{45} d_{25} =  d_{25}  d_{45} c_{24},~~c_{26} d_{63} d_{23} =  d_{23}  d_{63} c_{26},~~c_{46} d_{61} d_{41} =  d_{41}  d_{61} c_{46}.
$$

\medskip

One can see that
$$
\langle b_{13}, b_{15},  b_{35} \rangle \cong \langle c_{24}, c_{26}, c_{46} \rangle  \cong VP_3^+,
$$
$$
\langle b_{13}, b_{15},  b_{35},  c_{24}, c_{26}, c_{46} \rangle  \cong VP_3^+ \times VP_3^+.
$$

\begin{proposition} \label{hom}
There is a homomorphism of $\tau : Y_3 \to VP_6$ that is defined on the generators:
$$
b_{ij} \mapsto \lambda_{ij},~~c_{kl} \mapsto \lambda_{kl},~~d_{ij}  \mapsto \lambda_{ji}.
$$
\end{proposition}

\begin{proof}
The  image of $Y_3$ under $\tau$ is a subgroup of $VP_6$ which is generated by elements
\begin{align} \label{list}
\lambda_{13},~~\lambda_{15},~~\lambda_{35},~~\lambda_{24},~~\lambda_{26},~~\lambda_{46},~~
\lambda_{32},~~\lambda_{52},~~\lambda_{14},~~\lambda_{54},~~\lambda_{16},~~\lambda_{36}.
\end{align}
To prove that this map on the generators defines a homomorphism of $Y_3$
into $VP_6$ we need to prove that all relations of $Y_3$ go to relations of  $VP_6$ under $\tau$.
The Yang--Baxter relations go to  relations
$$
\lambda_{13} \lambda_{15} \lambda_{35} =  \lambda_{35} \lambda_{15} \lambda_{13},~~
\lambda_{24} \lambda_{26} \lambda_{46} =  \lambda_{46} \lambda_{26} \lambda_{24},
$$
which hold in $VP_6$. The commutativity relations go to  relations of the form
$$
\lambda_{ij} \lambda_{kl} = \lambda_{kl} \lambda_{ij},
$$ 
where all indexes are different. These relations also hold in $VP_6$.

The mixed relations go to  relations
$$
\lambda_{13} \lambda_{16} \lambda_{36} =  \lambda_{36}  \lambda_{16} \lambda_{13},~~
\lambda_{15} \lambda_{14} \lambda_{54}=  \lambda_{54} \lambda_{14} \lambda_{15},~~
\lambda_{35} \lambda_{32} \lambda_{52} =  \lambda_{52}  \lambda_{32} \lambda_{35},
$$
$$
\lambda_{24} \lambda_{54} \lambda_{52} =  \lambda_{52}  \lambda_{54} \lambda_{24},~~
\lambda_{26} \lambda_{36} \lambda_{32} =  \lambda_{32} \lambda_{36} \lambda_{26},~~
\lambda_{46} \lambda_{16} \lambda_{14} =  \lambda_{14}  \lambda_{16} \lambda_{46},
$$
that are relations of $VP_6$. 
\end{proof}

\begin{question}
Is it true that the homomorphism $\tau$ which was  constructed in Proposition~\ref{hom} is injective?
\end{question}

Let us define the elements 
$$
\Lambda_{12} = d_{41} b_{13} c_{24} d_{23}^{-1},~~\Lambda_{13} = d_{61} b_{15} c_{26} d_{25}^{-1},~~\Lambda_{23} = d_{63} b_{35} c_{46} d_{45}^{-1}
$$ 
in $Y_3$. As a corollary of Theorem \ref{t} we have 

\begin{corollary}
In the group $\langle \Lambda_{12}, \Lambda_{13}, \Lambda_{23} \rangle \leq Y_3$ the following relation holds
$$
\Lambda_{12} \Lambda_{13} \Lambda_{23} = \Lambda_{23} \Lambda_{13} \Lambda_{12} .
$$
\end{corollary}

If we denote the image $\tau (\Lambda_{ij})$ by $N_{ij}$, then
\bea
N_{12} = \lambda_{14} \lambda_{13} \lambda_{24} \lambda_{32}^{-1},~~N_{13} = \lambda_{16} \lambda_{15} \lambda_{26} \lambda_{52}^{-1},~~
N_{23} = \lambda_{36} \lambda_{35} \lambda_{46} \lambda_{54}^{-1}\nn
\eea
and these elements satisfy the YB relation
$$
N_{12} N_{13} N_{23} = N_{23} N_{13} N_{12}.
$$

\medskip

The group $Y_n$ is generated by three families of generators,
$$
b_{2i-1, 2j-1},~~c_{2i, 2j},~~1 \leq i < j \leq n,~~d_{2k, 2l-1},~~1 \leq k \not= l \leq n,
$$
and is defined by relations

-- {\it Yang-Baxter relations}
$$
b_{2i-1, 2j-1} b_{2i-1, 2k-1} b_{2j-1, 2k-1} =  b_{2j-1, 2k-1} b_{2i-1, 2k-1} b_{2i-1, 2j-1}, ~~~c_{2i, 2j} c_{2i, 2k} c_{2j, 2k} = c_{2j, 2k} c_{2i, 2k} c_{2i, 2j},
$$
where all indices $i, j, k$ are pairwise distinct;

-- {\it commutativity relations}
$$
b_{2i-1, 2j-1} c_{2k,2l} = c_{2k,2l} b_{2i-1, 2j-1},~~1 \leq i \not= j \leq n,~~1 \leq k \not= l \leq n,
$$
$$
b_{2i-1, 2j-1} d_{2k, 2l-1} = d_{2k, 2l-1} b_{2i-1, 2j-1},~~c_{2i,2j} d_{2k, 2l-1} = d_{2k, 2l-1} c_{2i,2j},
~~d_{2i,2j-1} d_{2k, 2l-1} = d_{2k, 2l-1} d_{2i,2j-1},
$$
where in the last three relations all indices are pairwise distinct;

-- {\it mixed   relations}
$$
b_{2i-1, 2j-1} d_{2k,2i-1} d_{2k,2j-1} = d_{2k,2j-1} d_{2k,2i-1}  b_{2i-1, 2j-1},~~c_{2i, 2j} d_{2j,2k-1} d_{2i,2k-1} = d_{2i,2k-1} d_{2j,2k-1}  c_{2i, 2j},
$$
where all indices $i, j, k$ are pairwise distinct.

In particular, for $n=2$ we have
$$
Y_2 = \langle b_{13}, c_{24}, d_{23}, d_{24}~|~b_{13} c_{24} = c_{24} b_{13}, ~~d_{23} d_{24} =  d_{24} d_{23} \rangle \cong \mathbb{Z}^2 * \mathbb{Z}^2.
$$

\begin{proposition}
\label{homvp6}
There is a homomorphism $\tau:Y_n\rightarrow VP_{2n}$ given by
\bea
\tau: b_{ij}\mapsto \lambda_{ij}; \qquad \tau: c_{ij}\mapsto \lambda_{ij};\qquad \tau: d_{ij}\mapsto \lambda_{ji}.\nn
\eea
\end{proposition}
The proof is straightforward.

The principal question of this section has the following answer:

\begin{proposition}
A Yang--Baxter set $(B \times C, \mathcal{R})$, where $\mathcal{R} = R_{41} R^B_{13} R^C_{24} R^{-1}_{23}$ defines a representation of $Y_n$, $n \geq 2$ in $\Sym\left((B \times C)^n \right)$.
\end{proposition}

\subsection{Infinitesimal  Lie algebras of $P_n$ }

The pure braid group $P_n$ is accompanied with the infinitesimal Lie algebra $L(P_n)$.
By analogy in  \cite{BEE} (see also \cite{BMVW}) it was defined a Lie algebra $L(VP_n)$ for the group $VP_n$. In this section we define a Lie algebra associated with the group $Y_n$.

 A Lie algebra  $L(Y_3)$ is generated by three families of elements:
$$
B_{13}, ~~B_{15}, ~~B_{35}, ~~C_{24},~~C_{26},~~C_{46},~~D_{23},~~D_{25}, ~~D_{41},~~D_{45}, ~~D_{61},~~D_{63},
$$
and is defined by relations:

-- {\it Yang-Baxter relations}
$$
[B_{13},  B_{15}] + [B_{13},  B_{35}] + [B_{15},  B_{35}] = 0, ~~ [C_{24},  C_{26}] + [C_{24},  C_{46}] + [C_{26},  C_{46}] = 0;
$$

-- {\it commutativity  relations}
$$
[B_{ij},  C_{kl}] = 0,~~B_{ij} \in \{ B_{13}, B_{15},  B_{35} \}, ~~C_{kl} \in \{ C_{24}, C_{26}, C_{46}\},
$$
$$
[B_{ij}, D_{kl}] = [C_{ij}, D_{kl}] = [D_{ij}, D_{kl}] =0,~\mbox{where all indices}~ i, j, k, l~\mbox{are different};
$$

-- {\it mixed   relations}
$$
[B_{13},  D_{61}] + [B_{13},  D_{63}] + [D_{61},  D_{63}] = 0, ~~[B_{15},  D_{41}] + [B_{15},  D_{45}] + [D_{41},  D_{45}] = 0,
$$
$$
[B_{35},  D_{23}] + [B_{35},  D_{25}] + [D_{23},  D_{25}] = 0,
$$
$$
[C_{24},  D_{25}] + [C_{24},  D_{45}] - [D_{25},  D_{45}] = 0, ~~[C_{26},  D_{23}] + [C_{26},  D_{63}] - [D_{23},  D_{63}] = 0
$$
$$
[C_{46},  D_{41}] + [C_{46},  D_{61}] - [D_{41},  D_{61}] = 0.
$$

From Theorem \ref{tc} it follows

\begin{corollary}
In the Lie algebra $L(Y_3)$ the elements
$$
E_{12} = B_{13} + C_{24} + D_{41} - D_{23},~~E_{13} = B_{15} + C_{26} + D_{61} - D_{25},~~E_{23} = B_{35} + C_{46} + D_{63} - D_{45},
$$
satisfy the CYBE,
$$
[E_{12},  E_{13}] + [E_{12},  E_{23}] + [E_{13},  E_{23}] = 0.
$$
\end{corollary}

In the previous subsection we construct a group homomorphism $\tau : Y_3 \to VP_6$. This homomorphism induces a homomorphism of Lie algebras.

\begin{proposition}
It exist a homomorphism of Lie algebras  $L(Y_3) \to L(VP_6)$ that is defined on the generators by the rules
$$
B_{ij} \mapsto L_{ij},~~C_{kl} \mapsto L_{kl},~~D_{ij}  \mapsto L_{ji}.
$$
\end{proposition}

\subsection{Simplicial sets and simplicial groups} Other approach for constructing set-theoretic solutions for the Braid equation on the direct product $B \times C$ is based on an operation of doubling of strings and is presented graphically  on Figure \ref{pic-S}.

\begin{figure}[H]
\centering
\includegraphics[width=60mm]{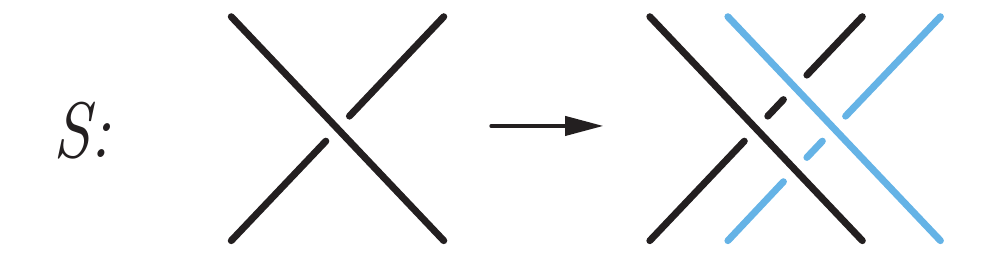}
\caption{Doubling of $S$}
\label{pic-S}
\end{figure}

 Recall the definition of simplicial groups (see \cite[p.~300]{MP} or \cite{BCWW}). A sequence of sets $X_* = \{ X_n \}_{n \geq 0}$  is called a
{\it simplicial set} if there are face maps:
$$
d_i : X_n \longrightarrow X_{n-1} ~\mbox{for}~0 \leq i \leq n
$$
and degeneracy maps
$$
s_i : X_n \longrightarrow X_{n+1} ~\mbox{for}~0 \leq i \leq n,
$$
that satisfy the following simplicial identities:
\begin{enumerate}
\item $d_i d_j = d_{j-1} d_i$ if $i < j$,
\item $s_i s_j = s_{j+1} s_i$ if $i \leq j$,
\item $d_i s_j = s_{j-1} d_i$ if $i < j$,
\item $d_j s_j = id = d_{j+1} s_j$,
\item $d_i s_j = s_{j} d_{i-1}$ if $i > j+1$.
\end{enumerate}
Here $X_n$ can be geometrically viewed as the set of $n$-simplices including all possible degenerate simplices.

A {\it simplicial group} is a simplicial set $X_*$ such that each $X_n$ is a group and all face and degeneracy operations are group homomorphism.

Let us define a simplicial set $\AB_*  = \{ AB_n \}_{n \geq 0}$, where $AB_n = B_{n+1}$ and we assume that $B_1$ is the trivial group.
There are face maps:
$$
d_i : AB_n = B_{n+1} \longrightarrow AB_{n-1} = B_{n } ~\mbox{for}~0 \leq i \leq n,
$$
that are deletion of the $(i+1)$-th strand 
and degeneracy  maps
$$
s_i : AB_n = B_{n+1} \longrightarrow AB_{n+1} = B_{n+2} ~\mbox{for}~0 \leq i \leq n,
$$
that are doubling  of the $(i+1)$-th strand. 
Hence, we have a simplicial set 
$$
\AB_* :\ \ \ \ldots\ \begin{matrix}\longrightarrow\\[-3.5mm] \ldots\\[-2.5mm]\longrightarrow\\[-3.5mm]
\longleftarrow\\[-3.5mm]\ldots\\[-2.5mm]\longleftarrow \end{matrix}\ B_4 \ \begin{matrix}\longrightarrow\\[-3.5mm]\longrightarrow\\[-3.5mm]\longrightarrow\\[-3.5mm]\longrightarrow\\[-3.5mm]\longleftarrow\\[-3.5mm]
\longleftarrow\\[-3.5mm]\longleftarrow
\end{matrix}\ B_3\ \begin{matrix}\longrightarrow\\[-3.5mm] \longrightarrow\\[-3.5mm]\longrightarrow\\[-3.5mm]
\longleftarrow\\[-3.5mm]\longleftarrow \end{matrix}\ B_2\ \begin{matrix} \longrightarrow\\[-3.5mm]\longrightarrow\\[-3.5mm]
\longleftarrow \end{matrix}\ B_1.
$$
It is easy to see that this simplicial set is not a simplicial group since the maps $d_i$ and $s_i$ are not group homomorphism.

\begin{example}
If we take the action of $s_0$ on
$$
B_3 = \langle \sigma_1, \sigma_2 ~||~ \sigma_1 \sigma_2 \sigma_1 =  \sigma_2 \sigma_1 \sigma_2 \rangle,
$$
then $s_0(\sigma_1) = e$ and $s_0(\sigma_2) = \sigma_2$. Hence, the relation under the action of $s_0$ takes the form $\sigma_2 = \sigma_2^2$, but it means that $\sigma_2 = e.$ 

If we take the action of $d_0$ on $B_3$, we get 
$$
d_0 (\sigma_1) = \sigma_2 \sigma_1,~~d_0 (\sigma_2) = \sigma_3,
$$
and relation takes the form
$$
 \sigma_2 \sigma_1 \cdot  \sigma_3 \cdot  \sigma_2 \sigma_1 =  \sigma_3 \cdot  \sigma_2 \sigma_1 \cdot  \sigma_3.
$$
But considering the homomorphism $B_4$ to $S_4$ we see that this relation is not true in $B_4$.
\end{example}

\begin{remark}
If one takes the pure braid groups instead of the braid groups then as was proved in~\cite{BCWW, CW} the corresponding simplicial set is  a simplcial group.
But in the pure braid group $P_3$ we have relations
$$
a_{12} a_{13} a_{23} = a_{23} a_{12} a_{13},~~a_{13} a_{23} a_{12} = a_{23} a_{12} a_{13},
$$
which are not the same as the Yang--Baxter relation.
\end{remark}

If we define the composition 
$$
D_n=s_0 s_1 \ldots s_{n-1} : B_n \to B_{2n},
$$
that is the doubling of all strands, where we take the composition from the  right to the  left, then   we get the next 
 analogous  of Corollary \ref{homvp3} in the braid group $B_{2n}$
\begin{lemma} \label{extb3}
The map 
$$
D_n=s_0 s_1 \ldots s_{n-1} : B_n \to B_{2n},
$$
is a group homomorphism and 
$$
D_n(\sigma_i)=\sigma_{2i}\sigma_{2i+1} \sigma_{2i-1}\sigma_{2i},~~i = 1, 2, \ldots,n-1.
$$
\end{lemma}

\begin{remark}
This statement can be illustrated by  Figure \ref{pic-PS}, where the generators of $B_n$ are realized as braidings of pairs of strands.
\end{remark}
\begin{figure}[h!]
\centering
\includegraphics[width=35mm]{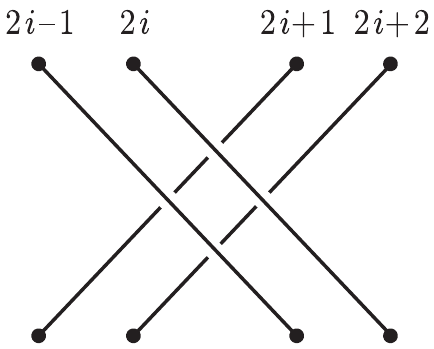}
\caption{Pairs of strands}
\label{pic-PS}
\end{figure}

\begin{example}
If we take the action of $D_3$ on $B_3$, we get the next subgroup of $B_6$
$$
D_3(B_3) = \langle \sigma_2 \sigma_1 \sigma_3 \sigma_2, ~~\sigma_4 \sigma_3 \sigma_5 \sigma_4~||~ \sigma_2 \sigma_1 \sigma_3 \sigma_2 \cdot \sigma_4 \sigma_3 \sigma_5 \sigma_4 \cdot \sigma_2 \sigma_1 \sigma_3 \sigma_2 =  \sigma_4 \sigma_3 \sigma_5 \sigma_4 \cdot \sigma_2 \sigma_1 \sigma_3 \sigma_2 \cdot \sigma_4 \sigma_3 \sigma_5 \sigma_4 \rangle.
$$
\end{example}

\begin{remark}
We defined the doubling $D_n$. By analogy we can define tripling and so on. More accurately  put
$$
D_n^{(k)} = s_0^k s_1^k \ldots s_{n-1}^k : B_n \to B_{(k+1)n},~~k = 1, 2, \ldots.
$$
In particular $D_n^{(1)} = D_n$.
\end{remark}

\subsection{Doubling of the pure virtual braid group.}
By using the same ideas as in the work~\cite{BCWW,CW} on the classical braids in \cite{BW1} it was introduced a simplicial group
$$
\VAP_* :\ \ \ \ldots\ \begin{matrix}\longrightarrow\\[-3.5mm] \ldots\\[-2.5mm]\longrightarrow\\[-3.5mm]
\longleftarrow\\[-3.5mm]\ldots\\[-2.5mm]\longleftarrow \end{matrix}\ VP_4 \ \begin{matrix}\longrightarrow\\[-3.5mm]\longrightarrow\\[-3.5mm]\longrightarrow\\[-3.5mm]\longrightarrow\\[-3.5mm]\longleftarrow\\[-3.5mm]
\longleftarrow\\[-3.5mm]\longleftarrow
\end{matrix}\ VP_3\ \begin{matrix}\longrightarrow\\[-3.5mm] \longrightarrow\\[-3.5mm]\longrightarrow\\[-3.5mm]
\longleftarrow\\[-3.5mm]\longleftarrow \end{matrix}\ VP_2\ \begin{matrix} \longrightarrow\\[-3.5mm]\longrightarrow\\[-3.5mm]
\longleftarrow \end{matrix}\ VP_1$$
on the pure virtual braid groups with $\VAP_n=VP_{n+1}$, the face homomorphism
$$
d_i : \VAP_n=VP_{n+1} \longrightarrow \VAP_{n-1}=VP_n
$$
given by deleting $(i+1)$-th strand for $0\leq i\leq n$, and the degeneracy homomorphism
$$
s_i : \VAP_n=VP_{n+1} \longrightarrow \VAP_{n+1}=VP_{n+2}
$$
given by doubling the $(i+1)$-th strand for $0\leq i\leq n$.

As in the case of braid group we can define the map
$$
D_n^{(k)} = s_0^k s_1^k \ldots s_{n-1}^k : VP_n \to VP_{(k+1)n},~~k = 1, 2, \ldots.
$$
In particular let us find the image of $VP_3$ under the action of $D_3 = D_3^{(1)}$. We can do it using the formulas of actions of $s_i$ as on the generators of $VP_3$. These formulas were found in 
\cite{BW}.

\begin{proposition}{\cite{BW}} \label{p3.1}
The degeneracy map $s_j : VP_n \longrightarrow VP_{n+1}$, $j = 0, 1, \ldots, n-1$, acts on the generators $\lambda_{k,l}$ and $\lambda_{l,k}$, $1 \leq k < l \leq n$, of $VP_n$ by the rules
$$
s_{i-1} (\lambda_{k,l}) = \left\{
\begin{array}{lr}
\lambda_{k+1,l+1} & for  ~i < k,\\
\lambda_{k,l+1} \lambda_{k+1,l+1} & for ~i = k, \\
\lambda_{k,l+1} & for  ~k < i < l,\\
& \\
\lambda_{k,l+1} \, \lambda_{k,l} & for ~i = l, \\
& \\
\lambda_{k,l} & for  ~i > l,
\end{array}
\right.
$$
$$
s_{i-1} (\lambda_{l,k}) = \left\{
\begin{array}{lr}
\lambda_{l+1,k+1} & for  ~i < k,\\
\lambda_{l+1,k+1} \lambda_{l+1,k} & for ~i = k, \\
\lambda_{l+1,k} & for  ~k < i < l,\\
& \\
\lambda_{l,k} \, \lambda_{l+1,k}   & for ~i = l, \\
& \\
\lambda_{l,k} & for  ~i > l.
\end{array}
\right.
$$
\end{proposition}

Using this Proposition we can find
$$
D_3(\lambda_{12}) = s_0 s_1 s_2(\lambda_{12}) =  s_0 s_1 (\lambda_{12}) = s_0 (\lambda_{13} \lambda_{12}) = \lambda_{14} \lambda_{24} \lambda_{13} \lambda_{23} = \lambda_{14}   \lambda_{13} \lambda_{24} \lambda_{23},
$$
and
$$
D_3(\lambda_{13}) = \lambda_{16} \lambda_{15} \lambda_{26} \lambda_{25},~~~D_3(\lambda_{23}) = \lambda_{36} \lambda_{35} \lambda_{46} \lambda_{45}.
$$
Since $D_3$ is a homomorphism $VP_3^+ \to VP_6^+$ then the elements $D_3(\lambda_{12})$,  $D_3(\lambda_{13})$ and $D_3(\lambda_{23})$ satisfy the Yang--Baxter   
relation
$$
D_3(\lambda_{12}) \, D_3(\lambda_{13})  \, D_3(\lambda_{23}) = D_3(\lambda_{12}) \, D_3(\lambda_{13}) \, D_3(\lambda_{23}).
$$

On the other hand we have
$$
D_3(\lambda_{12}) = D_3(\rho_1 \sigma_1) = \rho_2 \rho_1  \rho_3  \rho_2 \cdot \sigma_2 \sigma_1 \sigma_3 \sigma_2. 
$$
Hence, we constructed the group $Y_3$ and a homomorphism $\tau : Y_3 \to VP_6$
Under the canonical homomorphism $VP_6 \to VP_6^+$ which sends $\lambda_{ij}$ to  $\lambda_{ij}$ and sends $\lambda_{ji}$ to  $\lambda_{ij}^{-1}$ for $i < j$  the elements $N_{12}, N_{13}, N_{23}$ go to elements
$$
M_{12} = \lambda_{14} \lambda_{13} \lambda_{24} \lambda_{23},~~M_{13} = \lambda_{16} \lambda_{15} \lambda_{26} \lambda_{25},~~
M_{23} = \lambda_{36} \lambda_{35} \lambda_{46} \lambda_{45},
$$
of $VP_6^+$, correspondingly.  One can see that 
$$
D_3(\lambda_{12}) = M_{12},~~D_3(\lambda_{13}) = M_{13}, ~~D_3(\lambda_{23}) = M_{23}.
$$

\begin{corollary}
\label{homvp3}
In the group 
$\langle M_{12}, M_{13}, M_{23} \rangle \leq VP_6^+$ the following relation holds
$$
M_{12} M_{13} M_{23} = M_{23} M_{13} M_{12}.
$$
\end{corollary}
\begin{remark}
Proposition \ref{homvp6} is less restrictive than Corollary \ref{homvp3}, it does not demand the involutivity of the corresponding $R$-matrix. 
\end{remark}

\begin{remark}
This Corollary is equivalent to the statement of Lemma \ref{extb3}. Indeed, one can see that
$$
M_{12} = \lambda_{14} \lambda_{13} \lambda_{24} \lambda_{23} = \rho_3 \rho_2 \rho_1 \sigma_1 \rho_2 \rho_3 \cdot  \rho_2 \rho_1 \sigma_1 \rho_2 \cdot \rho_3 \rho_2  \sigma_2  \rho_3  \cdot \rho_2  \sigma_2 = \rho_2 \rho_1 \rho_3  \rho_2 \cdot \sigma_2 \sigma_1 \sigma_3 \sigma_2.
$$
On the other side,
$$
N_{12} = \lambda_{14} \lambda_{13} \lambda_{24} \lambda_{32}^{-1} = \rho_3 \rho_2 \rho_1 \sigma_1 \rho_2 \rho_3 \cdot  \rho_2 \rho_1 \sigma_1 \rho_2 \cdot \rho_3 \rho_2  \sigma_2  \rho_3  \cdot \sigma_2^{-1} \rho_2 = \rho_2 \rho_3 \rho_1  \rho_2 \cdot \sigma_2 \sigma_1 \sigma_3 (\rho_2 \sigma_2^{-1} \rho_2),
$$
and we see that this element can be constructed from $\lambda_{12}$ using more complicated operation than doubling of strings.
\end{remark}

Graphically the operation of constructing of set-theoretic solutions of the YBE on $B \times S$ is presented on Figure \ref{pic-R}.

\begin{figure}[H]
\centering
\includegraphics[width=60mm]{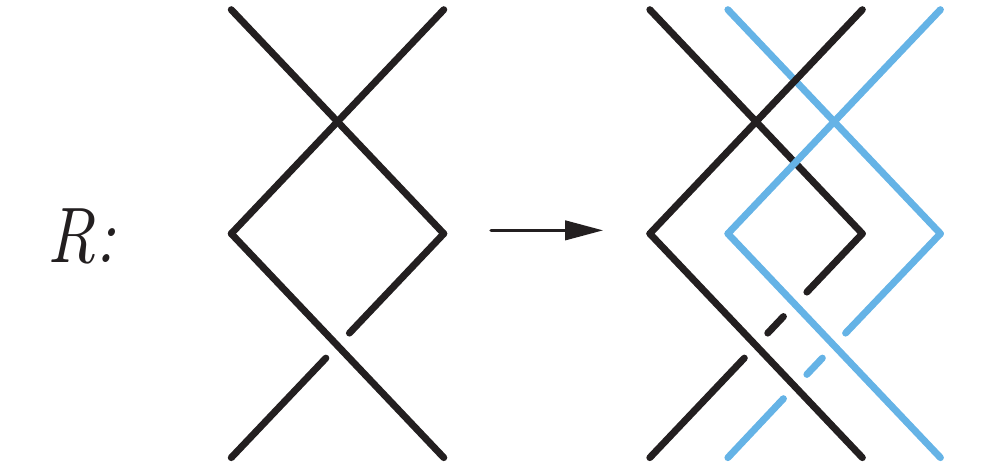}
\caption{Doubling of $R$}
\label{pic-R}
\end{figure}

\begin{ack}
The parts 1, 2  and 4 of the work was caried out with the support of the Russian Science Foundation grant 20-71-10110. The work on parts 3, 5 and 6 was supported by the Ministry of Science and Higher Education of Russia (agreement No. 075-02-2021-1392).
\end{ack}

\end{document}